\documentclass[12pt]{article}    
\usepackage{latexsym,amsfonts,amsmath,epsfig,tabularx,amsthm,dsfont}    
    
\theoremstyle{definition}     
    
\newtheorem{thm}{Theorem}

\newtheorem{prop}{Proposition}

\newcommand\reals{\mathbb{R}}    
\newcommand\R{{\mathbb{R}}}    
\newcommand\N{\mathbb{N}} 
\newcommand\Z{\mathbb{Z}} 
 
\newcommand\eps{\varepsilon}    
\newcommand\dist{{\rm dist}}

\newcommand{\APP}{\operatornamewithlimits{APP}} 
\newcommand{\INT}{\operatornamewithlimits{INT}} 
\newcommand{\OPT}{\operatornamewithlimits{OPT}} 
    
\newcommand{\e}{\varepsilon}    
\newcommand{\wt}{\widetilde} 
 
\newcommand{\vol}{{\rm vol}}    
     
\newcommand{\points}{\mathcal{P}} 
   
\newcommand{\ind}[1]{\mathds{1}_{#1}}    
\newcommand\dint{\,{\rm d}}

\newcommand{\C}{\mathcal{C}}

\renewcommand{\phi}{\varphi}

\newlength{\fixboxwidth}    
\title{Product rules are optimal for numerical integration\\
in  classical smoothness spaces} 
 
\author{
Aicke Hinrichs
\\ 
Institut f\"ur Analysis, Johannes Kepler Universit\"at\\ 
Linz, Austria,  
aicke.hinrichs@jku.at\\  
Erich Novak
\\ 
Mathematisches Institut, Universit\"at Jena\\ 
Ernst-Abbe-Platz 2, 07743 Jena, Germany, 
erich.novak@uni-jena.de\\ 
Mario Ullrich
\\ 
Institut f\"ur Analysis, Johannes Kepler Universit\"at\\ 
Linz, Austria,  
mario.ullrich@jku.at\\  
Henryk Wo\'zniakowski
\\ 
Department of Computer Science, Columbia University,\\ 
New York, NY 10027, USA, and\\ 
Institute of Applied Mathematics, University of Warsaw\\ 
ul. Banacha 2, 02-097 Warszawa, Poland, 
henryk@cs.columbia.edu}

\begin{document}    

\maketitle 

\centerline{\large{Dedicated to the memory of Joseph F.~Traub}}  
 
\medskip 

\begin{abstract}    
We mainly study numerical integration of real valued 
functions defined on the $d$-dimensional unit 
cube with all 
partial derivatives up to some finite order~$r\ge1$ bounded by one. 
It is well known that optimal algorithms that use $n$ function values 
achieve the error rate $n^{-r/d}$, where the hidden constant 
depends on $r$ and $d$. 
Here we prove explicit error bounds without hidden constants and,
in particular, show that the optimal order
of the error is $\min \bigl\{1, d \, n^{-r/d}\bigr\}$,
where now the hidden constant only depends on $r$, not on $d$. 
For $n=m^d$, this optimal order can be achieved by (tensor) product rules. 

We also provide lower bounds for integration defined over an arbitrary 
open domain of volume one.
We briefly discuss how lower bounds for integration may be applied for
other problems such as multivariate approximation and optimization. 
\end{abstract}     
 
\bigskip 
 
\section{Introduction}

Multivariate integration is nowadays a popular research problem
especially when the number of variables $d$ is huge. 
In this paper we mainly study numerical integration of $r\ge1$ times
continuously differentiable periodic and nonperiodic 
functions defined over the $d$-dimensional unit cube
whose             partial derivatives up to order $r\ge1$ are bounded by one.
Already in 1959, Bakhvalov \cite{B59} proved that 
the minimal number $n=n(\e,d,r)$ of function values 
which is needed to achieve an error at most $\eps>0$ 
satisfies 
$$
c_{d,r}\eps^{-d/r}\le n(\e,d,r)\le C_{d,r}\eps^{-d/r}
$$
for some positive $c_{d,r}$ and $C_{d,r}$ 
and the upper bounds are achieved by 
product rules. Note that for fixed $d$ and $r$ we have a sharp
behaviour with respect to $\e^{-d/r}$ and $n(\e,d,r)=\Theta(\e^{-d/r})$.

For large $d$, we would like to know how $c_{d,r}$ and $C_{d,r}$
depend on $d$. Unfortunately up to~2014, 
the knowledge on the dependence on $d$ was quite limited since 
the known lower bound on $c_{d,r}$ was exponentially small 
in $d$ whereas the known upper bound on $C_{d,r}$ 
was exponentially large in $d$. In \cite{HNUW12}, we proved for the
nonperiodic case that there exists a positive $c_r$ such that for all
$d$ and $\e\in(0,1)$ we have  
\begin{equation}\label{oldlowerbound}
n(\e,d,r)\ge c_r(1-\e)\,d^{\,d/(2r+3)}.
\end{equation}
Hence, we have a super-exponential dependence on $d$. This means 
that numerical integration suffers from the so-called
curse of dimensionality for fixed $r$. 

However, the exponent $d$ in \eqref{oldlowerbound} is 
$d/(2r+3)$, whereas in Bakhvalov's lower bound it is larger 
and equals       $d/r$. 
Furthermore, there is really no dependence on
$\e^{-1}$ in \eqref{oldlowerbound}, 
although we expect from Bakhvalov's bounds that it should be 
$\e^{-d/r}$.     

This is the point of departure of the current paper. We improve the
lower bound \eqref{oldlowerbound} and find a matching upper bound. 
Furthermore we will do it also for the periodic case which was not
studied in \cite{HNUW12}. The lower bound is found similarly as in
\cite{HNUW12} but instead of working with balls in the $\ell_2$-norm
we switch to balls in the $\ell_1$-norm which yields a better result.
The upper bound is achieved by product rules of $d$ copies of the
rectangle (or trapezoidal) quadrature  for the periodic case and of the
Gaussian quadrature for the nonperiodic case. 

We need a few definitions to formulate our results. 
We mainly study the problem of numerical integration, i.e., of approximating
the integral
\begin{equation}\label{eq:problem}
S_d(f) = \int_{[0,1]^d} f(x) \, \dint x
\end{equation}
for integrable functions $f\colon [0,1]^d\to\R$. 

The function class under consideration is the unit ball in the 
space of all $r$-times continuously differentiable functions 
on $[0,1]^d$, i.e.,
\[
\C^{\,r}_d \,=\, \{f\in C^{\,r}([0,1]^d)\colon \|D^\beta f\|_\infty\le 1 
        \,\text{ for all }\, \beta\in\N_0^d \,\text{ with }\, |\beta|_1\le r\}
\]
equipped with the norm 
$$
\|f\|_{\C^{\,r}_d} := \max_{\beta\colon |\beta|_1\le r} \|D^\beta
f\|_\infty.
$$ 
Here, $D^\beta$ denotes the usual (weak) partial derivative of 
order $\beta\in\N_0^d$.
Moreover, the sup-norm of a bounded function $f$ is given by 
$\|f\|_\infty=\sup_{x\in[0,1]^d}|f(x)|$.

We consider algorithms for approximating $S_d(f)$ that use finitely
many function values. More precisely, the general form 
of an algorithm
that uses $n$ function values is
$$
A_n(f)=\phi_n(f(x_1),f(x_2),\dots,f(x_n))\ \ \ \ \ 
\mbox{for all}\ \ \ f\in\C^{\,r}_d ,
$$
where $\phi_n:\R^d\to \R$ may be a nonlinear mapping and the sample
points $x_i\in[0,1]^d$ may be chosen adaptively, that is, the choice
of $x_i$ may depend on the already computed
$f(x_1),f(x_2),\dots,f(x_{i-1})$. Nonadaption means that the choice
of $x_i$ is independent of $f$, i.e., it is the same for all functions
$f$ from $\C^{\,r}_d$. Obviously even for the nonadaptive case, $x_i$
may depend on $n$ and the class $\C^{\,r}_d$. 

We consider the worst case setting 
in which the error of $A_n$ is defined as
$$
e(A_n)=\sup_{f\in\C^{\,r}_d }|S_d(f)-A_n(f)|.
$$
The $n$th minimal (worst case) error is given by
$$
e_n(\C^{\,r}_d )=\inf_{A_n}e(A_n),
$$
where the infimum is taken over all algorithms $A_n$, i.e., over all
mappings $\phi_n$ and adaptive choices of sample points
$x_1,x_2,\dots,x_n$ from $[0,1]^d$.

It is known by the result of Bahvalov \cite{B71} on adaption 
and the result of Smolyak \cite{S65} on nonlinear algorithms,
see also \cite{NW08} or \cite{TWW88}, 
that without loss of generality we may restrict ourselves to linear
algorithms and nonadaptive sample points, i.e., to algorithms of the
form
$$
A_n(f)=\sum_{i=1}^na_if(x_i)
$$    
for some real $a_i$ and some $x_i\in[0,1]^d$. Furthermore,
\begin{equation}\label{low1}
e_n(\C^{\,r}_d )=\inf_{a_i,x_i}\
\sup_{f\in\C^{\,r}_d }\big|S_d(f)-\sum_{i=1}^na_if(x_i)\big|
=\inf_{x_i}\ \sup_{f\in \C^{\,r}_d , \ f(x_1)=\cdots =f(x_n)=0}|S_d(f)|.
\end{equation}
Note that for $n=0$ we do not sample functions and 
$$
e_0(\C^{\,r}_d )=\|S_d\|=1.
$$

We can now formally define the minimal number of function values
needed to compute an~$\e$-approximation as 
$$
n(\e, d, r ) =\min\bigl\{n\colon\ e_n(\C^{\,r}_d )\le \e\,\bigr\}.
$$ 
Clearly, for $\e\ge1$ we have $n(\e, d, r ) =0$ and, therefore, we
always assume that $\e\in(0,1)$. 

We briefly discuss the results obtained in this paper 
and start with simplified results that might be easier to digest. 

\begin{thm} \label{thm:main}
For all $r\in\N$ there exist constants $c_{r,1},c_{r,2}>0$ such that for all 
$d\in\N$ and $\eps\in(0,1/2)$, we have
\begin{equation} 
c_{r,1}^d \left(\frac{d}{\eps}\right)^{d/r} 
\;\le\; n\bigl(\eps, d, r \bigr) 
\,\le\, c_{r,2}^d \left(\frac{d}{\eps}\right)^{d/r}. 
\end{equation} 
Moreover, the lower bound holds 
when $S_d$ is defined as integration over an arbitrary open set in
$\R^d$ of volume 1,
and the upper bound holds for arbitrary $\eps\in (0,1)$. 
\qed
\end{thm}

For the errors, the corresponding result is the following.

\begin{thm} \label{thm:main2}
For all $r\in\N$ there exist constants $c_{r,1},c_{r,2}>0$ 
such that for all $d,n\in\N$ 
with $n=m^d$ for some $m\in\N$, we have
\begin{equation} 
\min \bigl\{ {\textstyle\frac12}, c_{r,1} \, d\, n^{-r/d}\bigr\}
\;\le\; e_n \bigl(\C^r_d \bigr) 
\,\le\, \min \bigl\{ 1,  c_{r,2} \, d \,n^{-r/d}\bigr\}.
\end{equation} 
The lower bound holds for all $n\in\N$, 
whereas the upper bound has to be replaced by
$\min \bigl\{ 1, c_{r,2} \ d \, (n^{1/d}-1)^{-r}\bigr\}$ for general $n\in\N$.
  \qed
\end{thm}

Now we go more into the details and present bounds without any hidden constants. 
We start with 
the periodic case. To stress periodicity, 
we denote $n(\e,d,r)$ by $n^{\rm \,  per}(\e,d,r)$. For
$\e\in(0,d/(d+r)]$ we prove 
\begin{equation}\label{per1}
\frac{r}{r+d}\,\left(\frac{d}{d+r}\ \frac{1}{4^re^rr^{r-1}}\,
\frac{d}{\e}\right)^{d/r} 
\le n^{\rm \, per}(\e,d,r)\le
\left\lceil\left(\frac{1}{2\,(2\pi)^r}\,
\frac{d}{\e}\right)^{1/r}\right\rceil^d.
\end{equation}
In fact, the assumption that $\e\in(0,d/(d+r)]$ is only needed for the
lower bound, whereas the upper bound holds for all $\e\in(0,1)$.

We now turn to the nonperiodic case. 
Since the class of nonperiodic functions is larger than the class of
periodic functions we clearly have  
$n(\e,d,r)\ge n^{\rm \, per}(\e,d,r)$.
As already mentioned, to obtain upper bounds on $n(\e,d,r)$ 
we use product rules
based on univariate Gaussian quadratures. We use the error estimates of 
Gaussian quadratures proved by K\"ohler~\cite{K93} who studied the class $W_{\infty}^s$ of functions $f$ for which
$f^{(s-1)}$ is absolutely continuous and $\|f^{(s)}\|_\infty\le1$. 
Obviously, our class is a subset of $W_{\infty}^s$ for all $s\le r$
and we can apply  K\"ohler's
estimates which hold if the number of sample points is at least equal to
$s+1$. For the $d$-variate case, we use the result of Haber \cite{Ha70} for
product rules and obtain
\begin{equation}\label{nonper2}
n(\e,d,r)\le
\min_{s=1,2,\dots,r}\ \max\left\{(s+1)^d, 
\left\lceil\left(\frac{\pi}2\left(
\frac{e}{6\sqrt{3}}\right)^s\frac{d}{\e}\right)^{1/s}\right\rceil^{d}\right\}.
\end{equation}  
Obviously, for large $d/\e$ relative to $r$, more precisely, for 
$$
\frac{d}{\e}\ge \frac2{\pi}\left(\frac{6\sqrt{3}\,(r+1)}{e}\right)^r
$$ 
we have 
\begin{equation}\label{nonper3}
n(\e,d,r)\le
\left\lceil\left(\frac{\pi}2\left(
\frac{e}{6\sqrt{3}}\right)^r\frac{d}{\e}\right)^{1/r}\right\rceil^{d}.
\end{equation}  
In this case, the lower and upper bounds similarly depend on
$d$, $\e$ and are exponentially large in $d/\e$ and
exponentially small in $r$. This proves optimality
of product rules also for the nonperiodic case. 

The estimates \eqref{per1}--\eqref{nonper3} hold if the domain of
integration is $[0,1]^d$. Interestingly enough, our proof
technique for lower bounds works for more general integration domains.
We can integrate over an open set $D_d\subset\R^d$ of volume
one. Then, for the same smoothness class of real nonperiodic functions
defined over $D_d$,  we obtain a lower bound on the minimal number
$n(\e,D_d,r)$ of function values needed to compute an $\e$ approximation  
of the form
\begin{equation}\label{dd}
n(\e,D_d,r)\ge
\frac{r}{r+d}\,\left(\frac{d}{d+r}\ \frac{1}{6^re^rr^{r-1}}\,
\frac{d}{\e}\right)^{d/r}. 
\end{equation}
This is a similar lower bound as \eqref{per1} with $4^r$ replaced by
$6^r$. 

Assume first that $r$ is fixed. 
Then 
$n(\e,d,r)\asymp n^{\rm\,per}(\e,d,r)=\Theta((d/\e)^{d/r})$ which 
agrees with the bounds of Bakhvalov. We also have an exponential dependence
on $d$ which results in the curse of dimensionality, although it is 
delayed for large~$r$. Observe that the dependence of the lower and upper
bounds in \eqref{per1}--\eqref{dd} is exponentially small with respect to $r$. 

Now let $r$ be a function of $d$, i.e., $r=r(d)$. 
Although it is not a subject of this paper it
is easy to show that the curse of dimensionality still holds both in the nonperiodic and the periodic case
if $r(d)\,\ln\,r(d)=o(\ln\,d)$. It seems to be an interesting open problem 
to characterize how fast $r(d)$ must go to infinity with $d$ to break the curse
of dimensionality or to obtain various notions of tractability. We
are not sure if the bounds \eqref{per1} and \eqref{dd} 
are sufficiently sharp to solve such questions. 

It is well known that many multivariate problems are at least as hard
as multivariate integration. In particular, this holds for
multivariate approximation and optimization. That is why for the same
smoothness  class of $r$-times continuously differentiable functions 
we have the curse of dimensionality also for multivariate
approximation and optimization. Details are provided in Section
\ref{appopt}.


\section{Nonperiodic smooth functions}
In this section we study lower and upper bounds for numerical
integration of nonperiodic functions from the class 
$\C^{\,r}_d$. Lower bounds will be presented for more general
integration domains than $[0,1]^d$, whereas we show upper bounds only for 
the integration domain $[0,1]^d$.   

\subsection{Lower bounds} \label{sec:lb}

In this section we present a lower bound on $n(\e, d, r )$.
Interestingly enough, our proof technique can be applied not only for the
domain $[0,1]^d$ but also for an arbitrary open subset $D_d\subset \R^d$ of
volume ${\rm vol}_d(D_d)=1$. All definitions presented in the
previous section for $[0,1]^d$ (or, which is the same, for $(0,1)^d$)
readily generalize for $D_d$. In this case, we denote the class of
functions as $\C^{\,r}(D_d)$ with the norm
$$
\|f\|_{\C^{\,r}}=\max_{\beta:\,\|\beta|_1\le 1}\, \max_{x\in
  D_d}|(D^\beta f)(x)|.
$$
To stress the role of $D_d$ we denote the minimal number of function values
needed to compute an $\e$-approximation as $n(\e, D_d, r)$.

Let $\points=\{x_1,x_2,\dots,x_n\}\subset D^d$ 
be a collection of $n$ points. 
We construct a so-called fooling function $f$ from  
$\C^{\,r}(D_d)$ with $f(x_i)=0$, $i=1,2,\dots,n$, 
and as large as possible integral.
Due to \eqref{low1}, this allows us to get lower bounds on
$e_n(\C^{\,r}(D_d))$ and on $n(\e, D_d, r)$.

In fact, we will construct a fooling function $f$ from the Sobolev space
\[
W^{\,r+1}_\infty(D_d) \,=\, \{f: D_d\to\R\colon \ 
\|D^\beta f\|_\infty<\infty 
\,\text{ for all }\, \beta\in\N_0^d \,\text{ with }\, |\beta|_1\le r+1\}
\]
with $\|f\|_{\C^{\,r}}\le1$. 
Here $\|D^\beta f\|_\infty={\rm ess\,sup}_{x\in D_d}|(D^\beta f)(x)|$. 
Clearly, such a function $f$ belongs to $\C^{\,r}(D_d)$.

To define the fooling function for the given point set $\points$, 
we choose some $\varrho>0$, to be specified later,  and define the function 
\[ 
h_\varrho(x) \,=\, \min\left\{1, \frac{{\rm dist}(x,\points_\varrho)}{\varrho}\right\}
\qquad \mbox{for all}\quad x\in\R^d, 
\]
where for $A\subseteq \R^d$ we have   
\[
\dist(x,A) \,:=\, \min_{y\in A} \|x-y\|_1=
\min_{y\in A} \sum_{j=1}^d |x_j-y_j|
\]
and 
\[ 
\points_\varrho \,=\, \bigcup_{i=1}^n \bigl(\varrho B_1^d(x_i)\bigr). 
\] 
Here, $B_1^d(x_i)$ is the $\ell_1$-unit ball around $x_i$. 
We write $B_1^d$ for $B_1^d(0)$.

To treat functions of higher smoothness we consider the $r$-fold convolution 
of the function~$h_\varrho$ with the normalized indicator functions 
\begin{equation}\label{eq:gk} 
g_\varrho(x) \,=\, \frac{\ind{\varrho_r B_1^d}(x)}{\vol_d(\varrho_r B_1^d)}  
\,=\, \frac{1}{\vol_d(\varrho_r B_1^d)}\,\begin{cases} 
1 & \  \text{ if }\, x\in \varrho_r B_1^d,\\ 
0 & \  \text{ otherwise, } 
\end{cases} 
\end{equation} 
where $\varrho_r=\varrho/r$.
The convolution of a function $h$ with  $g_\varrho$  is given by 
\[ 
(h\ast g_\varrho)(x) \,=\, \frac{1}{\vol_d(\varrho_r B_1^d)}\,
\int_{\varrho_r B_1^d}  
h(x+t)\,\dint t   \qquad\mbox{for all}\quad x\in\R^d. 
\]  

The fooling function we consider in the sequel is therefore given by
\[
f_\varrho \,:=\, h_\varrho\ast \underbrace{g_\varrho\ast\dots\ast 
g_\varrho}_{r\text{-fold}}
\,:=\, h_\varrho\ast_r g_\varrho.
\]

It is clear that the support of the $r$-fold convolution of 
the function $g_\varrho$ is the $r$-fold Minkowski sum of the sets 
$\varrho_r B_1^d$, i.e.,  it is $\varrho B_1^d$.
This shows that the function $f_\varrho$ 
satisfies $f_\varrho(x)=0$ for all $x\in\points$. 

For the integral of $f_\varrho$, 
we only need to observe that $h_\varrho(x)=1$ 
if $\dist(x,\points)>2\varrho$ 
and that $\int_{\R^d}g_\varrho(x)\dint x=1$. 
This implies that $f_\varrho(x)=1$ if $\dist(x,\points)>3\varrho$, and hence 
\begin{equation}\label{eq:int}
\begin{split}
\int_{D_d} f_\varrho(x) \dint x 
\,&\ge\, \vol_d\left(D_d\setminus\points_{3\varrho}\right) 
\,\ge\, 1- \vol_d\left(\points_{3\varrho}\right) \\
\,&\ge\, 1- n\cdot \vol_d\left(3\varrho B_1^d(0)\right) 
\,=\, 1- n\,\frac{(6\varrho)^d}{d!} \\
\,&\ge\, 1- n \left(\frac{6e \varrho}{d}\right)^d.
\end{split}
\end{equation}
We stress that this bound holds for arbitrary collections of points 
$\points=\{x_1,x_2,\dots,x_n\}$ and sets~$D_d$. 
As we will see, $f_\varrho\in W^{\,r+1}_\infty(D_d)$ and the norm of the 
function $f_\varrho$ only depends on $\varrho$, $r$ and~$d$.
Since we are interested in a fooling function from the unit ball 
$\C^{\,r}(D_d)$ it remains to normalize $f_\varrho$. 
Hence we define
\begin{equation}\label{eq:fool}
f^*_\varrho(x) \,:=\, f_\varrho(x)/\|f_\varrho\|_{\C^{\,r}}  . 
\end{equation}
Using \eqref{eq:int} we obtain that $\int_{D_d}f^*_{\varrho}(x)\dint x 
\le \eps$ implies that 
\begin{equation}\label{eq:n_lower}
n \,\ge\, \left(1-\eps\cdot\|f_\varrho\|_{\C^{\,r}}\right)\,
\left(\frac{6e \varrho}{d}\right)^{-d}. 
\end{equation}
To finish our lower bound, we will choose $\varrho$ such that 
$\|f_\varrho\|_{C^k}\le \delta/\eps$ for some $\delta\in(0,1)$.

We bound the derivatives of $f_\varrho$ by induction. 
First of all note that $\|h_\varrho\|_\infty\le1$ and 
$\|D^\beta h_\varrho\|_\infty\le1/\varrho$ for $|\beta|_1=1$. 
Here, $D^\beta h_\varrho$ is the weak partial 
derivative of the Lipschitz-continuous function $h_\varrho$.
Let $e_j=(0,\dots,0,1,0,\dots,0)$ be the $j$th unit vector 
Then, for every $f\in L_\infty(\R^d)$, we have
\[\begin{split} 
D^{e_j}[f\ast g_\varrho](x) 
\,&=\, D^{e_j}\Bigl( \frac1{\vol_d(\varrho_r B_1^d)} \int_{\R^d}  
                f(x+t)\, \ind{\varrho_r B_1^d}(t) \,\dint t\Bigr) \\ 
&=\, \frac1{\vol_d(\varrho_r B_1^d)}\, D^{e_j}\Bigl(  
                \int_{e_j^\bot}  \int_{\R}  f(x+s+h e_j)\,  
                \ind{\varrho_r B_1^d}(s+h e_j) \,\dint h \,\dint s \Bigr) \\ 
&=\, \frac1{\vol_d(\varrho_r B_1^d)}\, \int_{e_j^\bot}\,  
                D^{e_j}\Bigl( \int_{\R}  f(x+s+h e_j)\,  
                \ind{\varrho_r B_1^d}(s+h e_j) \,\dint h \Bigr)\,\dint s, \\ 
\end{split}\] 
where $e_j^\bot$ is the hyperplane orthogonal to $e_j$.  
For any function $\wt f$ on~$\reals$ of the form 
\[  
\wt f(x) = \int_{x-a}^{x+a} g(y)\, \dint y  
\] 
with some continuous function $g$ we have 
\[ 
{\wt f}'(x) = g(x+a) - g(x-a). 
\] 
Therefore
\[\begin{split} 
\bigl|D^{e_j}[f\ast g_\varrho](x)\bigr| \,&=\, 
\Biggl|\frac1{\vol_d(\varrho_r B_1^d)}\,  
                \int_{(\varrho/r)B_1^d\cap e_j^\bot}\,  
\biggl[f\Bigl(x+s+h_{\rm max}(s)\,e_j
\Bigr)\, \\ 
&\qquad\qquad\qquad\qquad\qquad\qquad\qquad
-f\Bigl(x+s-h_{\rm max}(s)\,e_j
\Bigr) \biggr] 
\,\dint s \Biggr| \\
&\le\, \frac{2\,\vol_{d-1}(\varrho_r B_1^d\cap e_j^\bot)}
{\vol_d(\varrho_r B_1^d)}\,\|f\|_\infty
\,=\, \frac{2\,\vol_{d-1}(\varrho_r B_1^{d-1})}
{\vol_d(\varrho_r B_1^d)}\,\|f\|_\infty \\
\,&=\, \frac{d}{\varrho_r}\, \|f\|_\infty \,=\, \frac{d\,r}{\varrho}\, 
\|f\|_\infty
\end{split}\] 
with 
\[ 
h_{\rm max}(s) \,=\, \max\{h\ge0 \mid \ s+h e_j\in \varrho_r B_1^d\}. 
\] 
Moreover, using Young's inequality, we obtain
\[
\|D^{e_j}[f\ast g_\varrho]\|_\infty \,=\, \|(D^{e_j}f)\ast g_\varrho\|_\infty
\,\le\, \|(D^{e_j}f)\|_\infty.
\]
Using these inequalities recursively, 
we see that $f_\varrho\in W_\infty^{\,r+1}(D^d)$ with
\begin{eqnarray*}\label{prop1} 
 \Vert f_\varrho \Vert_\infty &\le& 1,\\ 
\label{prop2} 
 \forall{1\le\ell\le r+1}:\,\max_{\beta\in\N_0^d\colon |\beta|_1=\ell}\, 
        \|D^\beta f_\varrho\|_\infty &\le&  
        \varrho^{-\ell}\, \bigl(d\,r\bigr)^{\ell-1}. 
\end{eqnarray*}
This shows that 
\begin{equation}\label{eq:norm}
\|f_\varrho\|_{\C^{\,r}} 
\,\le\, \max\left\{1, \varrho^{-r}(d r)^{r-1} \right\},
\end{equation} 
if $\varrho\le d\,r$. Note that we simply ignore the bounds 
on $\|D^\beta f_\varrho\|_\infty$ 
for $|\beta|_1=r+1$. 
We now choose $\varrho =(\eps/\delta)^{1/r}(dr)^{1-1/r}$ to obtain 
$\|f_\varrho\|_{\C^{\,r}}\le\delta/\eps$ if $\eps\le\delta$.
This already implies the lower bound in our main result.

\begin{thm}\label{thm:lower}
For any $r,d\in\N$, $\delta\in(0,1)$ and $\eps\in(0,\delta]$ we have
$$
n(\eps, D_d , r)  \,\ge\, (1-\delta)\,c_r^d\,
\left(\frac{\delta\, d}{\eps}\right)^{d/r} \ \ \ \ 
\mbox{with}\ \ \ \  c_r=\frac1{6\,e\,r^{1-1/r}}. 
$$
Taking $\delta=d/(d+r)$, which maximizes the last bound, we obtain
$$
 n(\eps, D_d , r)   \,\ge\, 
\frac{r}{r+d}\,\left(\frac{d}{d+r}\ \frac{1}{6^re^rr^{r-1}}\,
\frac{d}{\e}\right)^{d/r}. 
$$
  \qed
\end{thm}

\begin{proof}
Using the construction of the fooling function with 
$\varrho=(\eps/\delta)^{1/r}(dr)^{1-1/r}$ as above, we obtain from 
\eqref{eq:n_lower} that 
\[
n \,\ge\, (1-\delta)\,\left(\frac{6e \varrho}{d}\right)^{-d} 
\,=\, (1-\delta) \left(6e\, r^{1-1/r}\right)^{-d} 
\left(\frac{\delta\, d}{\eps}\right)^{d/r}.
\]
Clearly,  the function $f(\delta)=(1-\delta)\delta^{d/r}$ is 
maximized for $\delta=d/(d+r)$ and substituting this~$\delta$ to the
previous bound we complete the proof.
\end{proof}
Note that the second bound on $n(\e, D_d , r)$ proves \eqref{dd}. 

\subsection{Upper bounds}

We describe a known upper error bound for numerical integration
for functions defined over the cube $[0,1]^d$ from the class
$\C^{\,r}_d=\C^{\,r}([0,1]^d)$,  
see \cite{Ha70}. 
We start with quadrature formulas~$Q_m$, $m \in \N$, 
for the univariate case $d=1$, 
$$
Q_m(f) = \sum_{i=1}^m a_i f(x_i) 
$$
with $a_i \in \R$,  $x_i \in [0,1]$ and $f:[0,1]\to\R$.
 
Then the (tensor) product rule $Q^{\,d}_m$ uses $m^d$ function values 
and  is defined by
$$
Q_m^d (f) = \sum_{i_1=1}^m \dots \sum_{i_d=1}^m a_{i_1} \dots a_{i_d} 
f(x_{i_1}, x_{i_2},\dots , x_{i_d}) ,
$$
where $f: [0,1]^d \to \R$. 

We now compare the worst case error 
$e(Q_m^{\,d}, \C^{\,r}([0,1]^d))$ 
with the error 
$e(Q_m, \C^{\,r}([0,1]))$. It is known that
\begin{equation}\label{upbo}
e(Q_m^{\,d}, \C^{\,r}([0,1]^d)) \le
\Big( \sum_{j=0}^{d-1} A^j \Big) \cdot 
e(Q_m, \C^{\,r}([0,1]))\ \ \ \ \mbox {with}\ \ \ \ A= \sum_{i=1}^m |a_i|. 
\end{equation}
Of course, we may use positive quadrature formulas
with $a_i \ge 0$ for which $A=1$.

For example, we may use the standard Gaussian formulas
for the class $\C^{\,s}([0,1])$ with $s=1,2,\dots,r$. 
Obviously, $\C^{\,r}([0,1])\subseteq\C^{\,s}([0,1])$. 
Then we obtain 
$$
e(Q_m, \C^{\,r}([0,1])) \le e(Q_m, \C^{\,s}([0,1])) \le 
c_s \,  m^{-s}, 
$$
where for $m>s$ we have 
$$
c_s=\frac{\pi}{2} \,  \left( \frac{e}{6\sqrt{3}}\right)^s. 
$$
This was proved in \cite{K93}, where 
the interval $[-1,1]$ was used instead of $[0,1]$. 
Hence, we need to rescale the problem
and multiply the estimate of \cite{K93} by $2^{-(s+1)}$. 
Observe that  $c_s$ is exponentially small in $s$, 
but this works only if $m>s$.  It is also known, see \cite{TWW88}
p.127, that the $m$th minimal error for algorithms that use not
necessarily positive coefficients
satisfies
$$
e_m(\C^{\,s}([0,1])=\frac{K_s}{(2\pi)^s}\,\frac1{m^r}(1+o(1)) \ \ \ \ \
\mbox{as}\ \ \ m\to\infty,
$$
where
$K_s=4/\pi\,\sum_{k=0}^\infty(-1)^{k(s+1)}\,(2k+1)^{-(s+1)}\in[1,\pi/2]$
is the Favard constant. 
Hence $(e/(6\sqrt{3}))^s=(0.26...)^s$ cannot be improved asymptotically
more than $(1/(2\pi))^s=(0.159\dots)^s$. 

For $s\in\{1,2\}$ we do not need to assume that $m\ge s$.
For $s=1$ we can take the optimal midpoint rule, see the discussion below.
For $s=2$ a better bound was proved in \cite{P93}. 

{}From \eqref{upbo} and the discussion of the Gaussian error bounds 
we obtain for $s=1,2,\dots,r$ and $m\ge s+1$ that
$$
e(Q_m^{\,d}, \C^{\,r}([0,1]^d)) \le
c_s \, d \, m^{-s}.
$$ 
Then $e(Q_m^{\,d}, \C^{\,r}([0,1]^d)) \le\e$ if we take
$$
m=m_s=\max\left\{ s+1,
\left\lceil\left(\frac{c_sd}{\e}\right)^{1/s}\right\rceil\right\}.
$$
Since $Q_{m_s}^{\,d}$ uses $m_s^d$ function values we have that
$n(\e,\C^{\,r}([0,1]^d)\ge \min_{s=1,2,\dots,r}m_s^d$.
This proves the following theorem.  
\begin{thm}
For any $r,d\in \N$ and $\e\in(0,1)$ we have
$$
n(\e,d,r)\le
\min_{s=1,2,\dots,r}\ \max\left\{ (s+1)^d, 
\left\lceil\left(\frac{\pi}2\left(
\frac{e}{6\sqrt{3}}\right)^s\frac{d}{\e}\right)^{1/s}\right\rceil^{d}\right\}.
$$
  \qed
\end{thm} 

Note that this proves \eqref{nonper2}. 
\vskip 1pc
It is interesting to notice that 
if we want to consider all $m\le r$ 
then to have an estimate $e(Q_m, \C^{\,r}([0,1])) \le c_r \,  m^{-r}$
we need to chose $c_r$ (super-)exponentially large in $r$. 
Indeed, take $m=2$ and 
consider the minimal worst case error $e_2$ 
of any two-point quadrature formula on the subset of 
the unit ball of $\C^{\,r}[0,1]$ 
consisting of polynomials of degree at most~4. 
Since for any two points $x_1,x_2$, a suitable polynomial 
$c (x-x_1)^2 (x-x_2)^2$ is in this unit ball for all $r$ 
(with  a positive $c$ independent of $r$), 
we find by a compactness argument that $e_2>0$. Hence,
$$ 
e_2 \le e(Q_2,\C^{\,r}([0,1])) \le c_r 2^{-r}
$$
implies $c_r \ge e_2 2^r$ for all $r$. Extending this argument 
to any fixed $m_0$ shows that there exist constants $e_{m_0}$ such
that 
$c_r \ge e_{m_0} m_0^r$ for all $r$.   

For $r=1$ we know more, see \cite{Su79}.
For $d=1$ we take the optimal midpoint rule $Q_m$ with error
$1/(4m)$   and obtain
$$
e(Q_m^d, W_\infty^1([0,1]^d)) \le \frac{d}{4}\,\frac1{n^{1/d}} 
$$
for $n=m^d$. This is not quite optimal, but almost. 
It is known that, asymptotically 
for large~$d$, the optimal constant is 
$d/(2e)$ instead of $d/4$. 


\section{Periodic smooth functions} 

In this section we study numerical 
integration over the domain $D_d=[0,1]^d$ 
for the classes of periodic functions. 
\[
\C^{\,r}_\pi  \,=\, \left\{f\in C^{\,r}(\R^d)\colon 
\ f \text{ is 1-periodic and }\|f\|_{C^r}\le1 \right\}.
\]
Hence, for all $f\in\C^{\,r}_\pi$ we have $f(x+e)=f(x)$
for all $x\in\R^d$ and all $e=(e_1,e_2,\dots,e_d)$ with
$e_j\in\{0,1\}$.
Since $\C^r_\pi\subset\C^r ([0,1]^d)$, 
all lower bounds for the class of periodic functions 
also hold for the class $\C^r ([0,1]^d)$.
We present slightly larger lower bounds and
smaller upper bounds for $\C^r_\pi$ compared to the results for the class  
$\C^r ([0,1]^d)$ that are provided in the last section.

\subsection{Lower bounds}
We will follow the same arguments as for nonperiodic functions.
For this let 
$$
\points=\{x_1,x_2,\dots,x_n\}\subset[0,1]^d
$$ 
be a collection of $n$ points. 
Since we want to construct a periodic fooling function, 
we consider the extended 
infinite point set $\wt\points:=\bigcup_{m\in\Z^d}(\points+m)$, and 
define for some $\varrho>0$ the function 
\[ 
\wt h_\varrho(x) \,=\, \min\left\{1, \frac{{\rm dist}(x,\wt\points_\varrho)}
{\varrho}\right\} 
\qquad \mbox{for all}\quad x\in\R^d, 
\]
where  
\[ 
\wt\points_\varrho \,=\, 
\bigcup_{m\in\Z^d}\bigcup_{i=1}^n \bigl(\varrho B_1^d(x_i+m)\bigr). 
\] 
Again, $B_1^d(x)$ is the $\ell_1$-unit ball in $\R^d$ around $x$ and 
we write $B_1^d$ for $B_1^d(0)$.
Clearly, $\wt h_\varrho$ is a 1-periodic function from $W^1_\infty(\R^d)$. 

Again, we consider the $r$-fold convolution of the function $\wt h_\varrho$
with the normalized indicator functions $g_\varrho$ from \eqref{eq:gk}, 
which we denote by $\wt f_\varrho$.
This function is obviously also 1-periodic and satisfies the same bounds on 
the derivatives as given in \eqref{eq:norm}.
However, the lower bound on the integral of $\wt f_\varrho$ over $[0,1]^d$ 
can be improved by the following argument. 

Let $g=g_\varrho\ast\dots\ast g_\varrho$ be the $r$-fold convolution 
of the functions $g_\varrho$ such that $\wt f_\varrho=\wt h_\varrho\ast g$. 
Then, using periodicity of $\wt h_\varrho$, we obtain
\[\begin{split}
\int_{[0,1]^d} \wt f_\varrho(x) \dint x 
\,&=\, \int_{[0,1]^d} \int_{\R^d}\wt h_\varrho(x-y)\, g(y) \dint y \dint x
\,=\, \int_{[0,1]^d} \wt h_\varrho(x) \dint x \int_{\R^d} g(y) \dint y \\
\,&=\, \int_{[0,1]^d} \wt h_\varrho(x) \dint x . 
\end{split}\]
Now we only use that $\wt h_\varrho$ satisfies $\wt h_\varrho(x)=1$ 
if $\dist(x,\wt\points)>2\varrho$ and 
obtain, similarly to \eqref{eq:int}, that
\[
\int_{[0,1]^d} \wt f_\varrho(x) \dint x 
\,\ge\, 1- n \left(\frac{4e\,\varrho}{d}\right)^d.
\]
Note that there is an improvement of $2^{-d}$ in the ``volume term''. 
Finishing the proof as for the nonperiodic case we obtain 
for any $\delta\in(0,1)$  and $\e\in(0, \delta]$
\[
n^{\rm per} (\eps, d, r)  \,\ge\, 
(1-\delta)\,c_r^d\,\left(\frac{\delta\, d}{\eps}\right)^{d/r} 
\]
with $c_r=1/(4e r^{1-1/r})$. 
We now choose $\delta $ to maximize the
lower bounds and obtain for $\delta=d/(d+r)$ the following
proposition. 
\begin{prop}
For any $r,d\in\N$, $\delta\in(0,1)$ and $\eps\in(0,\delta]$ we have
$$
n^{\rm per} (\eps, d, r)  \,\ge\, 
(1-\delta)\,c_r^d\,\left(\frac{\delta\, d}{\eps}\right)^{d/r} \ \ \ \ 
\mbox{with}\ \ \ c_r=\frac1{4e r^{1-1/r}}. 
$$
Taking $\delta=d/(d+r)$, which maximizes the last bound, we obtain
$$
n^{\rm per} (\eps, d, r)  \,\ge\, 
\frac{r}{r+d}\,\left(\frac{d}{d+r}\ \frac{1}{4^re^rr^{r-1}}\,
\frac{d}{\e}\right)^{d/r}.  
$$
  \qed
\end{prop}

Note that this proves the lower bound of \eqref{per1}. 

\subsection{Upper bounds} \label{sec:ub}

For the univariate case, it was proved by Motorny{\u\i} \cite{M73,M74},
see also \cite{TWW88} pp. 119--122, that the rectangle 
(or trapezoidal) quadrature 
$$
Q_m(f)=\frac1m\,\sum_{i=0}^{m-1}f\left(\frac{i}m\right)
\ \ \ \ \ \mbox{for all}\ \ f\in
\C^{\,r}_\pi
$$
is optimal for numerical integration over the class
$\widetilde{W}_{\infty}^r$ of periodic
functions whose $(r-1)$ derivatives are absolutely continuous and
$\|f^{(r)}\|_\infty\le1$. Furthermore, the $n$th minimal error 
is $1/(2(2\pi)^rn^r)$. The domain of integration considered by Motorny{\u\i} was $[0,2\pi]$. Therefore we have to rescale the problem to switch to $[0,1]$ which corresponds to multiplying the bound
of Motorny{\u\i} by $(2\pi)^{-(r+1)}$.

Since the class $\C^{\,r}_\pi$ is a subset 
of $\widetilde{W}_{\infty}^r$ this implies that  
$$
e_m(Q_m,\C^{\,r}_\pi)\le\frac1{2\,(2\pi)^r\,m^r}.
$$
For the $d$-variate case, we take the tensor product rule $Q_m^{\,d}$ 
of $d$ copies of $Q_m$ that uses $m^d$ function values and is defined by 
$$
Q_m^{\,d} (f) = \frac1{m^d}\,
\sum_{i_1=0}^{m-1} \dots \sum_{i_d=0}^{m-1}
f\left(\frac{i_1}{m}, \dots , \frac{i_d}{m}\right) \ \ \ \ 
\ \mbox{for all}\ \ \ f\in C^{\,r}_\pi.
$$
The upper  bound on the error of $Q_m^d$ follows again from
\cite{Ha70} and we obtain 
$$
e(Q_m^{\,d}, \C^{\,r}_\pi) \le d\,e(Q_m, \C^{\,r}_\pi)\le
\frac{d}{2\,(2\pi)^r\,m^r}.
$$ 
Equating the upper bound to $\e$ we get the following proposition.

\begin{prop}
For any $r,d\in\N$ and $\eps\in(0,1)$ we have
$$
n^{\rm per} (\eps, d , r)   \,\le\,
\left\lceil\left(\frac{1}{2\,(2\pi)^r}\,\frac{d}{\e}\right)^{1/r}\right\rceil^d.
$$
   \qed
\end{prop}
Note that this proves the upper  bound of \eqref{per1}.
 
\section{Multivariate approximation and optimization}\label{appopt}

We worked with the formula \eqref{low1} which says 
that, for the integration problem, 
the $n$th error bound 
$e_n (\C^r (D_d) )$ is given by
\begin{equation}\label{low11}
e_n(\C^r (D_d), \INT ) =    \inf_{x_1, \dots , x_n}\  \   \sup_{f\in \C^{\,r}(D_d)  , 
\ f(x_1)=\cdots =f(x_n)=0} \Big| \int_{D_d} f(x) \, {\rm d} x \Big|.
\end{equation}
Observe that we added the symbol $\INT$ since we now discuss two more problems, 
$\APP$ and $\OPT$. 
For 
the approximation problem we have 
$$
\APP (f) = S(f) = f
$$
with $S: \C^r (D_d)  \to L_\infty$
and we measure the error in the norm of $L_\infty$. 
Then we obtain the analogous formula 
\begin{equation}\label{low12}
e_n(\C^r (D_d) , \APP ) = \inf_{x_1, \dots , x_n}\  \   \sup_{f\in \C^{\,r}(D_d)  , \ 
f(x_1)=\cdots =f(x_n)=0} \Vert f \Vert_\infty .
\end{equation}
Clearly we obtain 
$$
e_n(\C^r (D_d), \INT  ) \le e_n(\C^r (D_d) , \APP) 
$$
and hence all lower bounds for integration also hold for the approximation 
problem. 

It is less known that also for the problem of 
(global) optimization, 
$$
S_d(f) = \OPT (f) = \sup_{x \in D_d}  f(x), 
$$
we get a very similar bound, namely
$$
\frac{1}{2} e_n (\C^r (D_d) , \APP) \le 
 e_n (\C^r (D_d) , \OPT) \le 
 e_n (\C^r (D_d) , \APP) , 
$$
see Wasilkowski~\cite{greg} and 
\cite{No88,NW10} for this and similar results. 
Hence all the lower bounds of this paper are also true 
(after a trivial modification because of the factor $1/2$) 
for the problem of global optimization. 
Actually it would not be difficult to improve the constants slightly 
for the problems $\APP$ and $\OPT$, but we do not go into the details.


\begin{thebibliography}{99}    
    
\setlength{\parsep }{-0.5ex}           
\setlength{\itemsep}{-0.5ex}           
           
\frenchspacing            
           
\newcommand\BAMS{\emph{Bull. Amer. Math. Soc.\ }}           
\newcommand\BIT{\emph{BIT\ }}           
\newcommand\Com{\emph{Computing\ }}           
\newcommand\CA{\emph{Constr. Approx.\ }}            
\newcommand\FCM{\emph{Found. Comput. Math.\ }}           
\newcommand\JAT{\emph{J. Approx. Th.\ }}           
\newcommand\JC{\emph{J. Complexity\ }}            
\newcommand\JMA{\emph{SIAM J. Math. Anal.\ }}           
\newcommand\JMAA{\emph{J. Math. Anal. Appl.\ }}           
\newcommand\JMM{\emph{J. Math. Mech.\ }}           
\newcommand\MC{\emph{Math. Comp.\ }}           
\newcommand\NM{\emph{Numer. Math.\ }}           
\newcommand\RMJ{\emph{Rocky Mt. J. Math.\ }}           
\newcommand\SJNA{\emph{SIAM J. Numer. Anal.\ }}           
\newcommand\SR{\emph{SIAM Rev.\ }}            
\newcommand\TAMS{\emph{Trans. Amer. Math. Soc.\ }}           
\newcommand\TOMS{\emph{ACM Trans. Math. Software\ }}           
\newcommand\USSR{\emph{USSR Comput. Maths. Math. Phys.\ }}             
           
\frenchspacing            
           
\addcontentsline{toc}{chapter}{Bibliography}            
 
\bibitem{B59} 
N. S. Bakhvalov, 
On the approximate calculation  of multiple  integrals, 
\JC {\bf 31}, 502--516, 2015 
[English translation; the original appeared in 
\emph{Vestnik MGU, Ser. Math. Mech. Astron. Phys. Chem}, 
\textbf{4}, 3--18, 1959.] 

\bibitem{B71} 
N. S. Bakhvalov,
On the optimality of linear methods for operator approximation in
convex classes of functions (in Russian),
\emph{Zh. Vychisl, Mat. Mat. Fiz.} {\bf  11}, 244-249, 1971.

\bibitem{Ha70}
S. Haber,
Numerical evaluation of multiple integrals,
\emph{SIAM Reviews} {\bf 12}, 481--526, 1970. 
 
\bibitem{HNUW12}
A. Hinrichs, E. Novak, M. Ullrich, H. Wo\'zniakowski,
The curse of dimensionality for numerical integration
of smooth functions, \MC {\bf 83}, 2853--2863, 2014.

\bibitem{K93}
P. K\"ohler,
Intermediate error estimates for quadrature formulas,
in: {Numerical integration, {IV} ({O}berwolfach, 1992)},
Internat. Ser. Numer. Math. \textbf{112}, 199--214, 
Birkh\"auser, Basel, 1993.

\bibitem{M73}
V. P. Motorny{\u\i},
The best quadrature formula of the form {$\sum_{k=1}^{n}p_{k}f(x_{k})$} for certain classes of periodic differentiable functions
(in Russian), \emph{Dokl. Akad. Nauk SSSR}, 221, 1060--1062, 1973,
[English transl.: Soviet Math. Dokl., 14, 1180--1183, 1973].

\bibitem{M74}
V. P. Motorny{\u\i},
The best quadrature formula of the form {$\sum_{k=1}^{n}p_{k}f(x_{k})$} for certain classes of periodic differentiable functions
(in Russian), \emph{Izv. Akad. Nauk SSSR Ser. Mat.}, 38, 583--614, 1974,
[English transl.: Math. USSR-Izv., 8, 591--620, 1974].

\bibitem{No88}           
E. Novak,            
\emph{Deterministic and Stochastic Error Bounds in          
Numerical Analysis},           
LNiM {\bf 1349}, Springer-Verlag, Berlin, 1988.            

\bibitem{NW08}           
E. Novak and H. Wo\'zniakowski,            
\emph{Tractability of Multivariate Problems},            
Volume I: Linear Information,     
European Math. Soc. Publ. House, Z\"urich,            
2008.            
 
\bibitem{NW10}    
E. Novak and H. Wo\'zniakowski,           
\emph{Tractability of Multivariate Problems},    
Volume II: Standard Information for Functionals,    
European Math. Soc. Publ. House, Z\"urich,            
2010.     

\bibitem{P93} 
K. Petras,  
Gaussian quadrature formulae, second Peano kernels, nodes, 
weights and Bessel functions, 
\emph{Calcolo}  \textbf{30}, 1--28, 1993.  

\bibitem{S65}
S. A. Smolyak,
On optimal restoration of functions and functionals of them
(in Russian), Candidate Dissertation, Moscow State University, 1965.

\bibitem{Su79} 
A. G. Sukharev,  
Optimal numerical integration formulas  
for some classes of functions of several variables, 
\emph{Soviet Math. Dokl.}  \textbf{20}, 472--475, 1979.  
 
\bibitem{TWW88}    
J. F. Traub, G. W. Wasilkowski and H. Wo\'zniakowski,    
\emph{Information-Based Complexity},     
Academic Press, 1988.     

\bibitem{greg} 
G. W. Wasilkowski, 
Some nonlinear problems are as easy as the approximation problem, 
\emph{Comput. Math. Appl.} \textbf{10}, 351--363, 1984. 
 
\end{thebibliography}
\end{document}